\documentclass[12pt]{amsart}
\usepackage{geometry}       
\usepackage{soul}
\usepackage{changes}
\usepackage{verbatim}
\usepackage{graphicx}
\usepackage{amssymb}
\usepackage{epstopdf}
\usepackage{mathrsfs}
\usepackage{subfigure}
\newtheorem{theorem}{Theorem}[section]
\newtheorem{lemma}[theorem]{Lemma}
\newtheorem{corollary}[theorem]{Corollary}

\theoremstyle{definition}

\newtheorem{proposition}[theorem]{Proposition}
\theoremstyle{remark}
\newtheorem{remark}[theorem]{Remark}

\numberwithin{equation}{section}

\def\R{\mathbb{R}}

\newcommand{\p}{\partial}
\newcommand{\wt}{\widetilde}

\newcommand{\BFR}{\mathbb{R}}

\newcommand{\ep}{\epsilon}

\newcommand{\I}{\mathbf{I}}

\newcommand{\be}{\begin{equation} }
	\newcommand{\ee}{\end{equation}}
\newcommand{\bse}{\begin{subequations}}
	\newcommand{\ese}{\end{subequations}}

\title[Boundary regularity of vortex patch]{ Boundary regularity of uniformly rotating vortex patches and an unstable elliptic
	 free boundary problem} 
\author{Yuchen Wang$^{1}$, Guanghui Zhang$^{2}$ and Maolin Zhou$^{3}$}
\address{ ${}^1$ School of Mathematics Science, Tianjin Normal University, Tianjin,  300387, China}
\email{wangyuchen@mail.nankai.edu.cn}

\address{${}^2$ School of Mathematics and Statistics, Huazhong University of Science and Technology, Wuhan, 430074, China}
\email{guanghuizhang@hust.edu.cn}

\address{${}^3$ Chern Institute of Mathematics and LPMC, Nankai University, Tianjin, 300071, China}
\email{zhouml123@nankai.edu.cn}
\thanks{${}^3$ partially supported by  National Key Research and Development Program of China (2021YFA1002400), Scientific Research Innovation Capability Support Project for Young Faculty（SRICSPYF-ZY2025172), NSF of China (12271437, 11971498), the Fundamental Research Funds for the Central Universities (Nankai University 63241642), Tianjin Natural Science Foundation Outstanding Youth Project (23JCJQJC00190) and Nankai Zhide Foundation.}
\thanks{${}^{1}$ partially supported by DFG ZH 605/1-2 and NSF of China (No.11831009,12471110). 
}
\thanks{${}^2$ partially supported by NSF of China (No. 12171176, 11971187). }
\date{\today}

\begin{document}
	
\begin{abstract}
{In this paper, we consider the sign-changing free boundary problem related to the uniformly rotating vortex patch solutions of the two-dimensional incompressible Euler equations. We prove that the boundary of the vortex patch locally forms a $90^\circ$ corner near singular boundary points, if the patch is Lipschitz.}
\\

\end{abstract}

\maketitle

\section{Introduction} \label{S:Intro}

Consider the {two-dimensional} incompressible Euler equations {on the plane}
\be \label{E:Euler-E}
\begin{cases}
	\p_t v+(v \cdot \nabla) v = - \nabla p,  \quad x \in \mathbb{R}^2, \\
	\nabla \cdot v = 0, \\
v \rightarrow 0, \quad |x| \to \infty,
\end{cases}
\ee 
where $v=(v^1,v^2)^t \in \mathbb{R}^2$ is the {velocity of the fluid and $p$ its pressure.Throughout this paper, we shall mainly work with the vorticity formulation of the two-dimensional incompressible Euler equations
\be \label{E:Euler-V} 
\begin{cases}
	\p_t \omega + (v \cdot \nabla) \omega = 0,& \quad x \in \BFR^2, \\
	v = -\nabla^{\perp}(-\Delta)^{-1} \omega, & \quad (a_1,a_2)^\perp := (-a_2,a_1)^t
\end{cases}
\ee 
with the vorticity field given by $\omega(t,x):= \p_1 v^2 - \p_2 v^1$, in which the velocity is uniquely determined by the vorticity due to the {\it stream-vorticity} formula.} Equation \eqref{E:Euler-V} is clearly a closed {PDE} system concerning $\omega$ only. {The global well-posedness of smooth solutions is well-established, as well as the Yudovich-type weak solutions $\omega \in L^1(\BFR^2) \cap L^\infty(\BFR^2)$}. {For necessary background and a review on the classical results concerning the Euler equations, we refer to \cite{MP1994, Yud1963} for interested readers.} 

{Particular interest} focuses on the vortex patch solutions of \eqref{E:Euler-V} in the form of
\be
\omega(t,x) = \sum_{j=1}^N \kappa_j \I_{D_j(t)}, 
\ee
where {$N$ is a positive integer, each $\kappa_1,\ldots,\kappa_N$ is a non-zero constant,} $D_1,\ldots,D_N$ are disjoint bounded domains and $\I_D$ denotes the characteristic function of the domain $D$.
The global existence and uniqueness of vortex patch solutions follow from \cite{Yud1963} straightforwardly. The global regularity of vortex patches was first proved by Chemin  \cite{Che1993} provided $\p D \in C^{k,\alpha}, k \geq 1,0<\alpha<1$, see other proofs given by Bertozzi and Constantin \cite{BC94}, and Serfati \cite{Serfati1994}. Recently, Kiselev and Luo \cite{KisLuo2023} established the global regularity of the patch if the initial data are of Sobolev classes.

On the other hand, Chemin \cite{Chemin1998} studied the {preserving of singularities on the boundary of an evolving patch, see also \cite{Elgindi2023} where Elgindi and Jeong considered the well/ill-posedness of specific singular structures.} The global regularity of the patch solution of a {two-dimensional} generalized transport equation was proved by Verdera in \cite{Verdera2023}.

Kiselev and Sverak \cite{Kis2014} studied the small-scale formation of the two-dimensional incompressible Euler equations, namely the infinite time blow-up solutions which satisfy $\|\nabla \omega(t)\|_{\infty} \sim C_1e^{e^{C_2t}}$. {Furthermore, Kiselev and Li \cite{KisLi2019} proved that an analogous phenomenon that the curvature of the boundary tends to infinity happens for vortex patches.} There is a long-standing conjecture that evolving vortex patches would eventually weakly converge to some steady vortex patches. This conjecture seems beyond the current capability of PDEs, but also indicates that the singular uniformly rotating vortex patches may occur as the global attractors in the class of vortex patch solutions. \\

Our main interest will focus on steady vortex patches. Being aware of the rotation and translation invariance of the system, our consideration mainly address the uniformly rotating vortex patches around the origin 
\[
\omega(t,x) = \omega_0(e^{-i\Omega t}x), 
\]
which are relative equilibria of the two-dimensional incompressible Euler equations, i.e., the vorticity configuration is invariant up to rigid motions, where $\Omega \in \BFR$ is the angular velocity and $\omega_0=\I_{D}$ is the initial patch. This setting holds due to the rotational invariance of $\mathbb{R}^2$ and conservation of the center of vorticity. {By the incompressible condition, there exists a scalar function $\Psi$}, referred to as the stream function, such that $v = -\nabla^\perp \Psi$. Then the uniformly rotating vortex patch is determined by a nonlinear elliptic problem concerning the relative stream function
{$\psi(x) = \Psi(x) + \frac{\Omega}{2}|x|^2$}. 

{We call $P \in \BFR^2$ is a stagnation point of the fluid if $v(P) =0$, or equivalently $\nabla \psi(P) =0$ in the rotating coordinate with the angular velocity $\Omega$}. In general $\psi$ is called the relative stream function, but we reserve the name stream function instead throughout this paper for simplicity. \\

Suppose $D$ is a bounded domain enclosed by a rectifiable Jordan curve.
It is not hard to verify that $\omega=\I_D$ is a uniformly rotating vortex patch if and only if $\psi$ solves the elliptic free boundary problem
\be \label{E:VP}
\begin{cases}
	- \Delta \psi = \I_D - 2 \Omega, \quad & x \in \BFR^2, \\ 
	\nabla(\psi - \frac{\Omega}{2}|x|^2) \rightarrow 0, \quad & |x| \rightarrow \infty, \\
	\psi = 0, \quad & x \in \p D,
\end{cases}
\ee
in which the vortical domain $D$ is the main unknown. 

{Denoted by $B_1$ the unit disk.} Clearly {the} circular patch $\omega = \I_{B_1}$ (the Rankine vortex) is a uniformly rotating vortex patch solution for any $\Omega \in \BFR$. Another explicit non-trivial uniformly rotating vortex patch is the Kirchhoff elliptic patch  
\[
\omega = \I_{E_{a,b}}, \quad E_{a,b}:=\left\{ (x_1,x_2) \in \BFR^2 \big|\; \frac{x_1^2}{a^2} + \frac{x_2^2}{b^2} \leq 1\;\right\}
\]
where the angular velocity satisfies $0< \Omega = \frac{ab}{(a+b)^2} \leq \frac{1}{4}$. 

We note that the vortical domains may have a complex topology. When the domain $D$ is not simply-connected, it is sufficient to slightly modify the equation \eqref{E:VP} to  
\be \label{E:VP-M}
\begin{cases}
	- \Delta \psi = \I_D - 2 \Omega, \quad & x \in \BFR^2, \\ 
	\nabla(\psi - \frac{\Omega}{2}|x|^2) \rightarrow 0, \quad & |x| \rightarrow \infty, \\
	\psi = c_j, \quad & x \in \Gamma_j,
\end{cases}
\ee
in which each closed curve $\Gamma_j$ is a connected component of $\p D$ and $c_j$ is a constant, $1 \leq j \leq n$, where $n$ is the number of connected-components of $\partial D$. Clearly the annular patches 
\[
\omega= \I_{C_{b,1}}, \; C_{b,1}:=\{ x \in \BFR^2 |\;\; b \leq |x| \leq 1\}
\]
are trivial solutions of \eqref{E:VP-M} holding for any $\Omega \in \BFR$ and {$b \in (0,1)$}. Beyond these explicit solutions, lots of uniformly rotating vortex patches are obtained via the local bifurcation approach \cite{Burbea1982, Hmidi2013,Hmidi2016a, Hoz2016a, Hmidi2016,Castro2016}. All of these solutions are sufficiently close to the explicit solutions, such as the circular/annular/{elliptic} vortex patches, in the sense of boundary perturbation.

Furthermore, Hassainia, Masmoudi, and Wheeler \cite{HMW2020} studied the global continuation of the $m$-fold local bifurcation curve $\mathscr{S}_m$ {parameterized by $s \in \BFR$}, emanating from the Rankine vortex for $m \geq 3$. They show that the end of bifurcation curves has the following alternatives:
\begin{enumerate}
	\item Smooth vortex patch with overhanging profiles, i.e., the boundary can not be parameterized as a graph in the polar coordinates.   
	\item  Limiting $V$-states, i.e., singular vortex patch where the boundary is parameterized as a non-$C^1$ graph in the polar coordinates.
\end{enumerate}
Together with the numerical observations given in \cite{Wu1984,Overman1986}, they made the following conjecture.\\

\noindent {\bf Conjecture:} {\bf The singular solution with $90^\circ$ corners seen in numerics for $m \geq 3$ exists as the (weak)  limits of patches along the $m$-fold branch $\mathscr{S}_m$ as $s \to \infty$}. \\

This conjecture is highly challenging since the nature of the problem essentially depends on the global geometric property of the special solutions emanating from the disk, which have not been obtained explicitly yet.  Some recent progress on the quantitative properties of the $m$-fold rotating vortex patch has been made by Park in \cite{Park2020} for sufficiently large $m$ via a variational argument based on the optimal transport. However, it seems not easy to gain uniform estimates on the higher-order derivatives of the graph function.

On the other hand, inspired by the remarkable progress on the study of the unstable obstacle problem \cite{ASW12, MG07}, it is an interesting problem to employ the techniques developed by the community of the free boundary problem to study the boundary regularity of uniformly rotating vortex patches, which motivates the present paper.  Since our approach follows a flavor of variation, the vortex patches we studied are not necessarily located on the bifurcation curves emanating from the unit disk.
 \\

Our consideration is based on the observation that the uniformly rotating vortex patch gives rise to the following elliptic unstable free boundary problem
\be \label{E:Sign-ch}
\begin{cases}
	-\Delta u  =  \lambda_{1}\mathbf{I}_{D} - \lambda_2 \mathbf{I}_{D^c} & \quad \text{ in } \mathbb{R}^2, \\
	u  =0 & \quad \text{ on } \p D.
\end{cases}
\ee
in which $\lambda_1,\lambda_2>0$ are prescribed constants and $D$ is an unknown bounded domain. By the maximum principle, clearly $u$ is positive in the domain $D$, while the sign of the solution $u$ may change in the complementary domain $D^c$. It is worth pointing out that equation \eqref{E:Sign-ch} indeed describes the simply-connected rotating vortex patch, or vortex patch solutions of \eqref{E:Euler-V} consist of simply-connected components. However, since our analysis is indeed local, namely we focus on the singular points on a connected component $\Gamma \subset \p D$, thus we can add a constant such that the stream function $u=0$ on $\Gamma$, then the conclusion also holds for non-simply vortical domains.

We shall focus on the regularity of the free boundary $\p D \subset \{u=0\}$ via tools developed in the study of obstacle-type problems since the pioneer work of Caffarelli \cite{Caf1977}, in particular, the Weiss-type monotonicity formula. See \cite{Figalli2019, Figalli2020} and references therein for more recent advances. To the best of the author's knowledge, there are only a few results on the regularity of uniformly rotating vortex patches except that they are sufficiently close to explicit solutions. In \cite{Hmidi2013} Hmidi, Mateu and Verdera proved that the boundary of uniformly rotating vortex patches on the local branch emanating from the Rankine vortex is $C^\infty$ smooth. The smoothness was improved to real-analytic very soon by Castro, C\'{o}rdoba and G\'{o}mez-Serrano in \cite{Castro2016}.
On the other hand, regarding the unstable obstacle problem, results obtained in \cite{ASW12, MG07} could not be implemented here straightforwardly since the level set $\{u=0\}$ does not completely coincide with the free boundary and behaviors of the stream function are essentially different, nor do the minimal/maximal solutions hardly exist.

Uniformly rotating vortex patch is a special class of the sign-changing free boundary problem if and only if $\Omega \in (0,\frac{1}{2})$. Here we mainly focus on the singularity occupying on the Lipschitz boundary of the vortex patch with $\Omega\in(0,\frac{1}{2})$. A domain is called a Lipschitz domain if its boundary is a Lipschitz curve. With the Lipschitz smoothness assumption on the boundary, we can give a complete description of the boundary regularity of uniformly rotating vortex patches.

The main result of this paper is given as follows:
\begin{theorem} \label{T:main-2}
Suppose $D$ is a Lipschitz domain and $\omega = \I_D$ is a uniformly rotating vortex patch with angular velocity $0<\Omega<\frac{1}{2}$. Then the singular set $S\subset\partial D$ contains at most finitely many points, and $\partial D\setminus S$ is smooth. Near each singular point $x\in S$, for sufficiently small $r>0$, $\partial D\cap B(x,r)$ consists of two $C^1$ arcs meeting at the right angle. In particular, if there is no stagnation point on the boundary, the boundary is $C^\infty$-smooth.
\end{theorem}

\begin{remark}
The conclusion also holds for the symmetric rotating vortex pairs by the same argument. It holds regardless of the solutions located on the bifurcation curves. The result recovers numerical observations obtained in \cite{Wu1984}. We expect the analysis to be useful in solving the conjecture on the limiting $V$-states, but a more delicate quantitative analysis on the geometry of the vortical domain seems necessary. 
\end{remark}

While Theorem \ref{T:main-2} suggests that certain singularities are restricted, others might persist in Lipschitz domains when $\Omega \notin (0,\frac{1}{2})$. However, a remarkable rigidity result by Gómes-Serrano, Park, Shi, and Yao \cite{GPSS2021} rules out another type of singularity.

\noindent {\bf Theorem A } { \it Let $D \subset \BFR^2$ be a bounded domain with Lipschitz boundary. Assume that $\omega = \I_D$ is a stationary/uniformly rotating vortex patch of \eqref{E:VP-M} for some $\Omega \in \BFR$. Then $D$ must be radially symmetric if $\Omega \in (-\infty,0] \cup [\frac{1}{2},+\infty)$ and radially symmetric up to a translation if $\Omega=0$. } \\	
	
This result has recently been extended to domains with  Jordan curve boundaries in \cite{FWZ2025}.



We would like to close the introduction with some remarks concerning the existence of singular vortex patches. While plenty of results are obtained, it is also worth pointing out that our consideration here is a priori, i.e., we assume the free boundary problem possesses a solution $(u, D)$.
 Whether or not a singular steady vortex patch really exists remains an open problem up to our best knowledge. Of course, one could apply a variational method to study the existence of rotating vortex patches as \cite{Tur83}, where the existence of uniformly rotating vortex patches could follow from maximizing the energy function penalized by the angular momentum 
\be \label{E:Energy}
E(\omega) = \frac{1}{2} \int_{\BFR^2} \omega (-\Delta)^{-1} \omega d \mu - \frac{\Omega}{2} \int_{\BFR^2} |x|^2 \omega d \mu  
\ee
among the rearrangement class of {prescribed vorticity and angular impulse.} But it seems highly non-trivial to determine the regularity of the boundary except if the vortex patch is highly concentrated. {Indeed, for steady vortex patches in a bounded domain, in \cite{Tur83} Turkington obtained the existence of steady vortex patches and $C^1$ regularity of the boundary of vortical domains when these vortex patches are highly concentrated. He also proposed a conjecture that the singular set on the boundary of a steady vortex patch has Hausdorff dimension zero if these patches are not concentrated.} So, if the conjecture in \cite{HMW2020} holds, it verifies the existence of singular uniformly rotating patches, which may be of independent interest. We also note that Garcia and Haziot \cite{Gar2022} studied the global continuation of the local branch emanating from a pair of corotating point vortices in a very recent work. More limiting $V$-states occur and their geometry is much more complicated if $D_1$ and $D_2$ are allowed to touch each other in several ways. \\

This paper is organized as follows. In Section \ref{S:Obstacle} we consider a sign-changing unstable elliptic free boundary problem and derive the Weiss-type monotonicity formula. The classification of singular points is obtained due to their blow-up limits. In Section \ref{S:Uni-Patch}, we prove the uniform regularity near singular points where $u$ has super-characteristic growth then finish the proof of Theorem \ref{T:main-2}. 


\section{A sign-changing two-phase unstable elliptic free boundary problem: Weiss-type monotonicity formula and classification of blow-up limits} \label{S:Obstacle}

The uniformly rotating vortex patch problem \eqref{E:VP} gives rise to the two-phase unstable elliptic free boundary problem as follows
\be \label{E:FreeBD-E}
\begin{cases}
-\Delta u  = \lambda_1 \I_{D} - \lambda_2 \I_{D^c}, & \quad \text{ in } \BFR^n,\; n \geq 2, \\
\quad \; \; u  = 0, & \quad \text{ on } \p D,
\end{cases}
\ee
where $\lambda_1,\lambda_2>0$ are prescribed constants and {the bounded domain $D$ is the main unknown.} Here
$D^c := \BFR^n \setminus D$ denotes the complementary set of the domain $D$.

Being aware that $u$ may change its sign in $\bar{D}^c$, in general $\bar{D}^c$ does not coincide with the set $\{u<0\}$.  Thus \eqref{E:FreeBD-E} is rather different from the classical unstable obstacle problem which has been extensively investigated since the seminal works \cite{ASW12,MG07}, that equation \eqref{E:FreeBD-E} does not only concern the function $u$ but also the domain $D$.

Assume equation \eqref{E:FreeBD-E} exists a weak solution $(u,D) \in H^1(\R^n) \times \mathbb{O}(\R^n)$, where $\mathbb{O}(\R^n):=\{ D \subset \R^n \text{ is a bounded domain} \}.$ By the standard elliptic regularity theory, we have $u \in C^{1,\alpha}(\BFR^2)$ for any $\alpha\in (0,1)$ and $u$ is real-analytic in either $D$ or $\bar{D}^c$. 

Employing the the implicit function theorem on a regular point $x \in \p D$, i.e., $|\nabla u(x)| \neq 0$, there exists $\delta>0$ such that $\p D \cap B_r(x)$ is $C^{1,\alpha}$ curve for some $0<\alpha<1$. Further employing the classical bootstrap argument, see for example  \cite{Kin1980}, the regularity of this local curve can be improved to real-analytic.

On the other hand, when $D$ is not $C^1$ smooth,
the singular set of $\p D$ defined by   
\be \label{E:Singular-set}
\mathcal{S}^u:= \left\{ x \in \p D \; \big| \;  |\nabla u(x)|  = 0 \right\},
\ee
and the geometric properties of the boundary nearby are our main interest. We begin with establishing the Weiss-type monotonicity formula for \eqref{E:FreeBD-E} as follows:
\begin{lemma} \label{L:Mono-F}
	Suppose that $(u, D)$ is a weak solution of \eqref{E:FreeBD-E}. For any $x_{0}\in\mathbb{R}^{n}$, let 
	\[
	\Phi_{x_{0}}(r)=r^{-n-2}\int_{B_{r}(x_{0})}|\nabla
	u|^{2}-2u(\lambda_{1}\I_{D}-\lambda_{2}\I_{D^{c}})
	d\mathcal{H}^{n}-2r^{-n-3}\int_{\partial
		B_{r}(x_{0})}u^{2}d\mathcal{H}^{n-1},
	\]
where $\mathcal{H}^k$ is the $k$-dimensional Hausdorff measure. Then for any $0<\rho<\delta$, the following Weiss-type monotonicity formula holds
	\begin{equation}\label{E:monoton}
		\Phi_{x_{0}}(\delta)-
		\Phi_{x_{0}}(\rho)=\int_{\rho}^{\delta}\int_{\partial
			B_{r}(x_{0})}2r^{-n-2}(\nabla u\cdot
		\nu-\frac{2u}{r})^{2}d\mathcal{H}^{n-1}dr\geq 0.
	\end{equation}
 \end{lemma}
 We begin with proving the following lemma.
\begin{lemma} 
If $u$ is a weak solution of \eqref{E:FreeBD-E}, then 
	\begin{equation} \label{E:Var-def}
		\int_{\mathbb{R}^{n}}|\nabla u|^{2} \mathrm{div} \vec{X} -2 (\nabla u)^t
		\mathcal{D} \vec{X} \nabla u- 2 u  (\lambda_{1}\I_{D}-\lambda_{2}\I_{D^{c}} )\mathrm{div} \vec{X} =0
	\end{equation}
	holds for all $\vec{X}=(X_1,\ldots,X_n) \in 
	C^{\infty}_{0}(\mathbb{R}^{n},\mathbb{R}^{n})$, {where $\mathcal{D} \vec{X} = (\p_i X_j)_{1 \leq i,j \leq n}$.} 
\end{lemma}
\begin{proof}
Since $\vec{X}$ is compactly supported  and $\nabla u \in L^2(\BFR^2)$, by a straightforward computation, we get
	\begin{equation*}
		\begin{aligned}
			0&=\int_{\mathbb{R}^{n}}\mathrm{div}(|\nabla u|^{2}\vec{X})\\
			&=\int_{\mathbb{R}^{n}}2 \sum_{i=1}^n\sum_{j=1}^n\p_{i} u \p_{ij} u X_j+|\nabla
			u|^{2}\mathrm{div}\vec{X} \\
			 & =\int_{\mathbb{R}^{n}} -2\sum_{i=1}^n\sum_{j=1}^n \left(\p_{ii} u \p_{j} u X_{j}+\p_{i} u \p_i X_{j} \p_{j} u\right)+|\nabla
			u|^{2}\mathrm{div}\vec{X}\\		
			& =\int_{\mathbb{R}^{n}}-2 \Delta
			u (\nabla u \cdot \vec{X}) -2(\nabla u)^t \mathcal{D} \vec{X} \nabla u +|\nabla
			u|^{2}\mathrm{div}\vec{X} \\
			&=\int_{\mathbb{R}^{n}} 2\lambda_{1}\I_{D} \nabla u \cdot \vec{X} - 2\lambda_{2}\I_{D^{c}} \nabla u \cdot \vec{X} -2(\nabla u)^t \mathcal{D}\vec{X} \nabla u+ |\nabla
			u|^{2}\mathrm{div}\vec{X}\\
			& = \int_{\BFR^n}  |\nabla
			u|^{2}\mathrm{div}\vec{X} -2(\nabla u)^t \mathcal{D}\vec{X} \nabla u -2u(\lambda_{1}\I_{D} - \lambda_{2}\I_{D^{c}}) {\rm div} \vec{X}.
		\end{aligned}
	\end{equation*}
Then the proof is complete.

\end{proof}
Now we are going to prove the Weiss-type monotonicity formula of the sign-changing elliptic free boundary problem.
\begin{proof}[Proof of Lemma \ref{L:Mono-F}]

Suppose that $r,k$ are positive constants and let
\[
\eta_{k}(x)=\max\left\{0,\min\left\{1,(r-|x|)k\right\}\right\},
\vec{X}_{k}(x)=\eta_{k}(x)x=\left(\eta_{k}(x)x_{1},\cdots,\eta_{k}(x)x_{n}\right).
\]
A straightforward computation shows that
\[
\begin{split}
 \mathrm{div}\vec{X}_{k}(x) & =\sum_{i=1}^n \p_i \left(\eta_{k}(x)x_{i}\right)
		=n \eta_{k}(x)+\nabla \eta_{k}(x) \cdot x, \\
	(\nabla u)^t \mathcal{D} \vec{X}_{k} \nabla u  & = \sum_{i=1}^n\sum_{j=1}^n \p_i u \p_j (\eta_{k}(x)x_{i}) \p_j u\\
		&=\sum_{i=1}^n\sum_{j=1}^n \p_i u \eta_{k}\delta_{ij} \p_j u +\p_i u\p_j (\eta_{k})x_{i}\p_j u. \\
		&
	=|\nabla u|^{2}\eta_{k}+(\nabla u\cdot x)(\nabla u\cdot \nabla\eta_{k}).
	\end{split}
\]
According to \eqref{E:Var-def}, one has
\begin{equation*}
	\begin{aligned}
		0&=\int_{\mathbb{R}^{n}} n|\nabla u|^{2}\eta_{k}+|\nabla
		u|^{2}\nabla \eta_{k}\cdot x-2|\nabla u|^{2}\eta_{k}-2(\nabla
		u\cdot x)(\nabla u\cdot \nabla\eta_{k})\\
		&-\int_{\mathbb{R}^{n}}2(\lambda_{1}\I_{D} -  \lambda_2 \I_{D^c}) (n \eta_{k}+\nabla
		\eta_{k}\cdot x)u.
	\end{aligned}
\end{equation*}
Due to the definition of cut-off function $\eta_k(x)$, we get $\eta_{k} = 1, |x| \leq r-\frac{1}{k}$ and $\eta_k =0, |x| \geq r$. Now letting $k\to +\infty$, we obtain
\begin{equation} \label{E:Weak-solu}
	\begin{aligned}
		0=&\int_{B_{r}(0)} (n-2) |\nabla u|^{2}- 2n u(\lambda_{1}\I_{D}-\lambda_{2}\I_{D^{c}}) d\mathcal{H}^{n}\\
		&-r \int_{\partial B_{r}(0)}|\nabla u|^{2}-2(\nabla u\cdot
		\nu)^{2}- 2u(\lambda_{1}\I_{D} -\lambda_{2}\I_{D^{c}}) d\mathcal{H}^{n-1},
	\end{aligned}
\end{equation}
for $a.e.\;r>0$, where $\nu$ is the unit outward normal vector of $B_{r}(0)$. Then one has 
\[
	\begin{split}
		\frac{d}{dr}\Phi_{x_{0}}(r)& =-(n+2)r^{-n-3}\int_{B_{r}(x_{0})}|\nabla
		u|^{2}-2u(\lambda_{1}\I_{D}-2\lambda_{2}\I_{D^{c}})
		d\mathcal{H}^{n}\\
		& +r^{-n-2}\int_{\partial B_{r}(x_{0})}|\nabla
		u|^{2}-2u(\lambda_{1}\I_{D}-\lambda_{2}\I_{D^{c}}) 
		d\mathcal{H}^{n-1}\\
		&  +8r^{-5}\int_{\partial
			B_{1}(0)}u^{2}(x_{0}+rx)d\mathcal{H}^{n-1}-4r^{-4}\int_{\partial
			B_{1}(0)}u(x_{0}+rx) \left(\nabla u(x_{0}+rx)\cdot x\right)d\mathcal{H}^{n-1}.
\end{split}
\]
Inserting \eqref{E:Weak-solu} into above equation, one has  
\be \label{E:E1}
\begin{aligned}
	\Phi_{x_0}^{\prime}(r)	&=-r^{-n-3}\int_{B_{r}(x_{0})}(n-2)|\nabla
		u|^{2}-2n u(\lambda_{1}\I_{D}-\lambda_{2}\I_{D^{c}}) d\mathcal{H}^{n}\\
		& -4r^{-n-3}\int_{B_{r}(x_{0})}|\nabla
		u|^{2}-u(\lambda_{1}\I_{D}-\lambda_{2}\I_{D^{c}})
		d\mathcal{H}^{n}\\
		& +r^{-n-2}\int_{\partial B_{r}(x_{0})}|\nabla
		u|^{2}-2u(\lambda_{1}\I_{D}-\lambda_{2}\I_{D^{c}}) d\mathcal{H}^{n-1}\\
		& +8r^{-n-4}\int_{\partial
			B_{r}(x_{0})}u^{2}d\mathcal{H}^{n-1}-4r^{-n-3}\int_{\partial
			B_{r}(x_{0})}u(\nabla u\cdot \nu)d\mathcal{H}^{n-1}\\
		&=2 r^{-n-2}\int_{\partial B_{r}(x_{0})} (\nabla u\cdot
		\nu)^{2}d\mathcal{H}^{n-1}
		 +8r^{-n-4}\int_{\partial
			B_{r}(x_{0})}u^{2}d\mathcal{H}^{n-1}\\
		& -8r^{-n-3}\int_{\partial
			B_{r}(x_{0})}u(\nabla u\cdot \nu)d\mathcal{H}^{n-1}\\  
		&=2r^{-n-2}\int_{\partial B_{r}(x_{0})}\left(\nabla u\cdot \nu-\frac{2u}{r}\right)^{2}d\mathcal{H}^{n-1}, 
	\end{aligned}
\ee
in which we used the identity
\[
\begin{split}
 \int_{\partial
	B_{r}(x_{0})}u(\nabla u\cdot \nu)d\mathcal{H}^{n-1}&= \int_{B_r(x_0)} |\nabla u|^2 +u \Delta u \;d\mathcal{H}^{n}\\
& = \int_{B_r(x_0)} |\nabla u|^2 -u(\lambda_1 \I_D - \lambda_2 \I_{D^c}) d \mathcal{H}^n. 
\end{split}
\]
Equation \eqref{E:monoton} is obtained by integrating equation \eqref{E:E1}.
\end{proof}

Lemma \ref{L:Mono-F} leads to a classification of the blow-up limits of $u$ at singular points on the free boundary.
\begin{proposition}[Classification of blow-up limits] \label{P:Class}
	Suppose $u$ is a solution of \eqref{E:FreeBD-E} and $x_{0}\in \mathcal{S}^u$. The following cases occur alternatively:
\begin{enumerate} 
\item In the case of $\Phi_{x_{0}}(0+)=-\infty$, one has $\lim_{r\to
		0}r^{-3-n}\int_{\partial B_{r}(x_{0})}u^{2}d\mathcal{H}^{n-1}=+\infty$,
	and each limit of
	\[
	v_r(x)=\frac{u(x_{0}+rx)}{S(x_{0},r)}
	\]
	as $r\to 0$ is a homogeneous harmonic polynomial of degree two, where
	\[
	S(x_{0},r):=\left(r^{1-n}\int_{\partial
		B_{r}(x_{0})}u^{2}d\mathcal{H}^{n-1}\right)^{1/2}.
	\]

\item In the case of $\Phi_{x_{0}}(0+) > -\infty$, for any given $R>0$, there exists $\tau(R)>0$ such that for $r<\tau(R)$,
	\[
	u_{r}(x):=\frac{u(x_{0}+rx)}{r^{2}}
	\]
	is bounded in $W^{1,2}(B_{R}(0))$. Let $\{r_{j}\}_{j\in\mathbb{N}}$ be a sequence such that
	$\{u_{r_{j}}\}_{j\in\mathbb{N}}$ converges to $u_{0}$ weakly in
	$W^{1,2}_{loc}(\mathbb{R}^n)$, then $u_{0}$ is a degree two homogeneous solution of {
\begin{equation}\label{eq:vp-limit}	
		\begin{cases}
			-\Delta u_{0}=\lambda_{1}\mathbf{I}_{D_*^1}-\lambda_{2}\mathbf{I}_{D_{*}^{2}}  & \text{ in }\mathbb{R}^{n},\\
			\quad \quad u_0   =0 & \text{ on }  \partial D_{*}^1 \cup \partial D_*^2,
		\end{cases}
		\end{equation}
where $D_*^1$ and $D_*^2$ are open cones satisfying $D_*^1\cap D_*^2=\emptyset$, $\mathbb{R}^2\setminus(D_*^1\cup D_*^2)\subset\{u_0= 0\}$ and $u_0>0$ in $D_*^1$. Here $\lambda_1\mathbf{I}_{D_{*}^1}-\lambda_2\mathbf{I}_{D_{*}^2}=\lim_{j\to \infty}\left( \lambda_1\mathbf{I}_{D_{r_{j}}}-\lambda_2\mathbf{I}_{D_{r_{j}}^c}\right)$ in the sense of distribution and $D_{r}=\{x:x_{0}+rx\in D\}$. }
\end{enumerate}
\end{proposition}

\begin{remark}
{When $n=2$ and $D_*^1\neq \emptyset$, we get $|\mathbb{R}^2 \setminus (D_*^1 \cup D_*^2) |=0$. Indeed, direct calculation shows that for all $x\notin D_*^1 \cup D_*^2\cup\{0\}$, $|
\nabla u(x)|\neq 0$. Therefore, $u_0$ is a solution to the equation
		\[
		-\Delta u_0=\lambda_1\I_{D_*^1}-\lambda_2\I_{(D_*^1)^c} \text{ in } \mathbb{R}^2
		\]
		with $u_0>0$ in $D_*^1$ and $u_0=0$ on $\partial D_*^1$. }
\end{remark}

\begin{proof} [Proof of Proposition \ref{P:Class}]
	(1). If $\Phi_{x_{0}}(0+)=-\infty$, we suppose towards a contradiction that $\limsup_{r\to
		0}S(x_{0},r)\leq Mr^2$ for some constant $M$, there exists an $r_{0}>0$
	such that, for $r<r_{0}$,
	\[
	\int_{\partial B_{r}(x_{0})}u^{2}d\mathcal{H}^{n-1}\leq M^{2}r^{n+3}.
	\]
	Thus
	\begin{equation*}
		\begin{aligned}
			\int_{B_{r}(x_{0})}|u|d\mathcal{H}^{n}&=\int_{0}^{r}\int_{\partial
				B_{t}(x_{0})}|u|d\mathcal{H}^{n-1}dt\\
			&\leq C_{1}\int_{0}^{r}\left(\int_{\partial
				B_{t}(x_{0})}u^{2}d\mathcal{H}^{n-1}\right)^{\frac{1}{2}}t^{\frac{n-1}{2}}dt\\
			&\leq C_{2}\int_{0}^{r}t^{\frac{n+3}{2}}t^{\frac{n-1}{2}}dt\\
			&\leq C_{3}r^{n+2}.
		\end{aligned}
	\end{equation*}
	Consequently
	\[
	\left|\frac{1}{r^{n+2}}\int_{B_{r}(x_{0})}2u\left(\lambda_{1}\I_{D}-\lambda_{2}\I_{D^{c}}\right) \;d\mathcal{H}^{n}\right|\leq C_{4}.
	\]
	We conclude that
	\[
	\begin{split}
	\Phi_{x_{0}}(r)& \geq -\left|r^{-n-2}\int_{B_{r}(x_{0})}2u\left(\lambda_{1}\I_{D}-\lambda_{2}\I_{D^{c}}\right)
	d\mathcal{H}^{n}\right|-2r^{-n-3}\int_{\partial
		B_{r}(x_{0})}u^{2}d\mathcal{H}^{n-1}\\
		& \geq -C_{4}-2M^{2},
	\end{split}
	\]
	for all $r<r_{0}$, which contradicts the assumption that
	$\Phi_{x_{0}}(0+)=-\infty$. Therefore
	$\limsup_{r\to 0}\frac{S(x_{0},r)}{r^2}=+\infty$. We pick up a sequence
	$\{r_{j}\}_{j\in\mathbb{N}}$ such that $\lim_{j\to\infty}r_{j}=0$
	and $\lim_{j\to \infty}\frac{S(x_{0},r_{j})}{r_j^2}=+\infty$. Let
	
	\[
	v_{k}(x)=\frac{u(x_{0}+r_{k}x)}{S(x_{0},r_{k})}.
	\]
	Then $\|v_{k}\|_{L^{2}(\partial B_{1}(0))}=1$. We may assume that
	$\lim_{k\to\infty}v_{k}\to v_{0}$ 
	weakly in $L^{2}(\partial B_{1}(0))$. 
	
	Let $K=\frac{\max\{\lambda_{1},\lambda_{2}\}}{2n}$. Then
	$u_{r}^{+}+K|x|^{2}$ and $u_{r}^{-}+K|x|^{2}$ are subharmonic
	functions which imply
	\[
	\frac{1}{t^{n-1}}\int_{\partial
		B_{t}(0)}u_{r}^{+}+K|x|^{2}d\mathcal{H}^{n-1},
	\]
	\[
	\frac{1}{t^{n-1}}\int_{\partial
		B_{t}(0)}u_{r}^{-}+K|x|^{2}d\mathcal{H}^{n-1},
	\]
	are monotone increasing in $t$. Therefore
	\begin{equation*}
		\begin{aligned}
			\int_{\partial B_{t}(0)}u_{r}^{+}d\mathcal{H}^{n-1}&\leq
			t^{n-1}\int_{\partial
				B_{1}(0)}u_{r}^{+}+K|x|^{2}d\mathcal{H}^{n-1}\\
			&\leq C_{5}t^{n-1}\left(\|u_{r}\|_{L^{2}(\partial
				B_{1}(0))}+1\right),
		\end{aligned}
	\end{equation*}
	for some constant $C_{5}$. {Thus we get}
	\[
	\int_{B_{1}(0)}u_{r}^{+}d\mathcal{H}^{n}=\int_{0}^{1}\int_{\partial
		B_{t}(0)}u_{r}^{+}d\mathcal{H}^{n-1}dt\leq \frac{C_{5}}{n}\left(\|u_{r}\|_{L^{2}(\partial
		B_{1}(0))}+1\right).
	\]
Analogously, we have 
	\[
	\int_{B_{1}(0)}u_{r}^{-}d\mathcal{H}^{n}\leq \frac{C_{6}}{n}\left(\|u_{r}\|_{L^{2}(\partial
		B_{1}(0))}+1\right).
	\]

By the monotonicity formula
	\eqref{E:monoton} we have
	\[
	\int_{B_{1}(0)}\left(|\nabla
	u_{r_{k}}|^{2}+2\lambda_{1}\I_{D_{r_{k}}}u_{r_{k}}-2\lambda_{2}\I_{D_{r_{k}}^{c}}u_{r_{k}} \right)d\mathcal{H}^{n}-2\int_{\partial
		B_{1}(0)}u_{r_{k}}^{2}d\mathcal{H}^{n-1}\leq \Phi_{x_{0}}(r_{1}).
	\]
	Let $T_{k}=\frac{S(x_{0},r_{k})}{r_{k}^{2}}$ and $
		v_k(x)=\frac{u(x_{0}+r_kx)}{S(x_{0},r_k)}$. Then
	\[
	\int_{B_{1}(0)}|\nabla v_{k}|^{2}d\mathcal{H}^{n}\leq
	T_{k}^{-2}\Phi_{x_{0}}(r_{1})+CT_{k}^{-2}\int_{B_{1}(0)}|u_{r_{k}}|d\mathcal{H}^{n}+2\int_{\partial B_{1}(0)}v_{k}^{2}d\mathcal{H}^{n-1}.
	\]
	Note that $\lim_{k\to\infty}T_{k}=+\infty$ and
	\[
	T_{k}^{-1}\int_{B_{1}(0)}|u_{r_{k}}|d\mathcal{H}^{n-1}\leq
	CT_{k}^{-1}\left(\|u_{r_{k}}\|_{L^{2}(\partial B_{1}(0))}+1\right)\leq C(1+T_{k}^{-1}).
	\]
	Thus
	\[
	\sup_{k\in\mathbb{N}}\|\nabla v_{k}\|_{L^{2}(B_{1}(0))}\leq C_{6},
	\]
	for some constant $C_{6}$. By the Poincar\'e-Steklov inequality (see
	e.g. \cite[Lemma 3.30]{ErnLuc2004}) we have
	\[
	\| v_{k}\|_{L^{2}(B_{1}(0))}\leq C_{7}\left(\|\nabla
	v_{k}\|_{L^{2}(B_{1}(0))}+\|v_{k}\|_{L^{2}(\partial
		B_{1}(0))}\right)\leq C_{8},
	\]
	where $C_{8}$ does not depend on $k$. Thus we may assume that
	$v_{k}$ converges to $v_{0}$ weakly in $W^{1,2}(B_{1}(0))$.    
	Letting $k\to\infty$, we obtain that $v_{0}$ is harmonic and
	\[
	\int_{B_{1}(0)}|\nabla v_{0}|^{2}d\mathcal{H}^{n}\leq
	2\int_{\partial B_{1}(0)}v_{0}^{2}d\mathcal{H}^{n-1}.
	\]
	Since $|\Delta v_{k}|\leq \max\{\lambda_{1},\lambda_{2}\}$, by the
	elliptic regularity and the Sobolev embedding theorem, for any given
	$\alpha\in (0,1)$, there exists a constant $C_{9}$ such that
	\[
	\sup_{k\in \mathbb{N}}\|v_{k}\|_{C^{1,\alpha}(B_{1}(0))}\leq C_{9}.
	\]
	We may assume that $v_{k}$ converges to $v_{0}$ in
	$C^{1,\beta}(B_{1}(0))$ for some $\beta\in (0,\alpha)$. Notice that
	$v_{k}(0)=0$ and $\nabla v_{k}(0)=0$, we have $v_{0}(0)=0$ and $\nabla
	v_{0}(0)=0$. It follows from the frequency formula (see
	e.g. Lemma 4.2 of \cite{MG07}) that $v_{0}$ is a harmonic polynomial
	of degree 2.\\
	
	(2). If $\Phi_{x_{0}}(0+)>-\infty$, there exists a constant $M<\infty$
	such that 
	\[
	\limsup_{r\to
		0}r^{-3-n}\int_{\partial B_{r}(x_{0})}u^{2}d\mathcal{H}^{n-1}<M.
	\]
	Thus we can choose $r_{0}>0$ small enough such that
	$\|u_{r}\|_{L^{2}(\partial B_{1}(0))}\leq M+1$ for $r\in
	(0,r_{0})$. Without loss of generality, we set $R=1$. Similar to case (1), we have
	$\int_{B_{1}(0)}|u_r|d\mathcal{H}^{n}< C_{10}$ for some constant
	$C_{10}$. It follows from the monotonicity formula that there exists a
	constant $C_{11}$ such that
	\[
	\int_{B_{1}(0)}|\nabla u_{r}(x)|^{2}\leq
	\Phi_{x_{0}}(r)+2\max\{\lambda_{1},\lambda_{2}\}\int_{B_{1}(0)}|u_{r}|d\mathcal{H}^{2}+2\int_{\partial
		B_{1}(x)}u_{r}^{2}d\mathcal{H}^{n-1}\leq C_{11},
	\]
	for all $r\in (0,r_{0})$. Thus $\{u_{r}\}$ is uniformly bounded in
	$W^{1,2}(B_{1}(0))$. We may choose a sequence
	$\{r_{k}\}_{k\in\mathbb{N}}$ such that, for some $u_{0}\in
	W^{1,2}(B_{1}(0))$, $u_{r_{k}}\to u_{0}$ weakly in $W
	^{1,2}(B_{1}(0))$, $u_{r_{k}}\to u_{0}$ in $L^{2}(B_{1}(0))$ and
	$u_{r_{k}}\to u_{0}$ in $L^{2}(\partial B_{1}(0))$. Moreover, for any
	given $\sigma>\rho>0$, we have
	\[
	\int_{\rho}^{\delta}\int_{\partial B_{r}(x_{0})}2\left(\nabla
	u_{r_{k}}\cdot\nu-\frac{2u_{r_{k}}}{r}\right)^2 d\mathcal{H}^{n-1}dr=\Phi_{x_{0}}(r_{k}\delta)-\Phi_{x_{0}}(r_{k}\rho)\to 0,
	\]
	as $k\to \infty$. Letting $k\to \infty$, we have
	\[
	\nabla u_{0}(x)\cdot x-2u_{0}(x)=0
	\]
	for all $x\in\mathbb{R}^{n}$, and thus $u_{0}$ is a homogeneous function of
	degree 2.

	On the other hand, $u_{r_k}$ is a weak solution to equation
	\[
	-\Delta u_{r_k}=\lambda_1 \I_{D_{r_k}}-\lambda_2 \I_{(D_{r_k})^c},
	\]
	with $u_{r_k}=0$ on $\partial D_{r_k}$, where $D_{r_k}=\{x\in\R^n|x_0+r_kx\in D\}$.
	By the weak-* compactness of $L^{\infty}$, we may assume that $\lambda_1 \I_{D_{r_k}}-\lambda_2 \I_{(D_{r_k})^c}$ weak-* converges to some functions $\nu(x) \in L^\infty(\R^n)$, 
	and $u_0$ is a weak solution to equation 
	\begin{equation*}
		-\Delta u_{0}=\nu.
	\end{equation*}
	By the elliptic regularity theory, for any given $\alpha\in(0,1)$, $\{u_{r_k}\}$ is bounded in $C^{1,\alpha}_{loc}(\R^n)$. Thus we may assume $u_{r_k}$ converges to $u_0$ in $C^{1,\beta}_{loc}(\R^n)$ for some $\beta\in(0,\alpha)$. We first show that in each connected component of $\{u_0\neq 0\}$, either $-\Delta u_0\equiv \lambda_1$ or $-\Delta u_0\equiv -\lambda_2$. Indeed, if $u_0(x_1)<0$, there exists an $r_0$ such that $u_0<0$ in $B_{2r_0}(x_1)$. By the $C^{1,\alpha}$ convergence of $u_{r_k}$, there exists a constant $N\in \mathbb{N}$ such that for $k\geq N$, $u_{r_k}<0$ in $B_{r_0}(x_1)$. Thus $-\Delta u_{r_k}=-\lambda_2$ in $B_{r_0}(x_1)$ for $k\geq N$. Let $v_k=u_{r_k}-\frac{\lambda_2}{2n}|x|^2$. Then $\{v_k\}_{k>N}$ are harmonic functions and $v_k\to u_0-\frac{\lambda_2}{2n}|x|^2$ in $C^{1,\alpha}(B_{r_0}(x_1))$. It follows that $u_0-\frac{\lambda_2}{2n}|x|^2$ is harmonic in $B_{r_0}(x_1)$, which implies $-\Delta u_0=-\lambda_2$ in $B_{r_0}(x_1)$. Therefore, $\nu\equiv -\lambda_2$ in $\{u_0<0\}$. If $u_0(x_2)>0$, there exists an $r_0$ such that $u_0>0$ in $B_{2r_0}(x_2)$. By the $C^{1,\alpha}$ convergence of $u_{r_k}$, there exists a constant $N\in \mathbb{N}$ such that for $k\geq N$, $u_{r_k}>0$ in $B_{r_0}(x_2)$. Note that for each $k\geq N$, $-\Delta u_{r_k}=\lambda_1$ or $-\Delta u_{r_k}=-\lambda_2$ in $B_{r_0}(x_2)$. Choosing a subsequence if necessary, we may assume that $-\Delta u_{r_k}=\lambda_1$ or $-\Delta u_{r_k}=-\lambda_2$ for all $k\geq N$. Thus $-\Delta u_0=\lambda_1$ or $\Delta u_0=-\lambda_2$ in $B_{r_0}(x_2)$. Therefore, in each connected component of $\{u_0>0\}$, either $\nu\equiv \lambda_1$ or $\nu\equiv -\lambda_2$.

Next, we show that $\nu(x)=0$ for almost all $x\in\{u_0=0\}$. Since the set $\{u_0>0\}$ is $C^{1,\alpha}$-smooth, it is easy to see that the Hausdorff dimension of $\{u=0\}\cap\{|\nabla u|\neq 0\}$ is $n-1$, so the Lebesgue measure of $\{u_0=0\}\cap\{|\nabla u_0|\neq 0\}$ is 0. In $\{u_0=0\}\cap\{|\nabla u_0|=0\}$ we have $D^2 u_0(x)=0$ a.e., which implies that $\Delta u_0(x)=0$ for almost all $x\in\{u_0=0\}\cap\{|\nabla u_0|=0\}$. Therefore, $\Delta u_0(x)=0$ for almost all $x\in\{u_0=0\}$.\\
    
	Since $u_0$ is homogeneous of degree 2, $\nu(rx)=\nu(x)$ for all $x\neq 0$ and $r>0$. Therefore $\nu=\lambda_1\I_{D_*^1}-\lambda_2 \I_{D_*^2}$, where $D_*^1$ and $D_*^2$ are open cones, 
	$D_*^1\cap D_*^2=\emptyset$, $\mathbb{R}^n\setminus(D_*^1\cup D_*^2)\subset\{u_0= 0\}$, $u_0>0$ in $D_*^1$, and $u_0=0$ on $\partial D_*^1\cup\partial D_*^2$.   
\end{proof} 
	 
Proposition \ref{P:Class} implies the singular points on the free boundary would be classified into the following two classes:  
\begin{enumerate}
	\item singular points where $u$ has super-characteristic growth (the blow-up limit I), 
	\item singular points where $u$ has characteristic growth (the blow-up limit II).
\end{enumerate}


\section{Singularities of the two-dimensional unstable elliptic free boundary problem } \label{S:Uni-Patch}

The proof of our main results will be accomplished by the following lemmas.  Firstly, we focus on the two-dimensional unstable elliptic free boundary problem 
\be\label{E:VP-1}
\begin{cases}
	-\Delta \psi  = \lambda_1\I_{D} - \lambda_2 \I_{D^c} & \quad \text{ in } \BFR^2, \\
	\nabla( \psi - \frac{\Omega}{2}|x|^2)  \rightarrow 0, & \quad |x| \rightarrow \infty, \\
	\quad \; \; \psi  = 0 & \quad \text{ on } \p D,
\end{cases}
\ee
where $D$ is a bounded domain. Due to the maximum principle, we have $\psi > 0  \text{ in } D$. 
\subsection{Blow-up analysis} 
Suppose $x_0 \in \mathcal{S}^\psi$. It follows from the monotonicity formula 
\be \label{E:Mono-2}
\Phi_{x_0}^\psi(r) = \frac{1}{r^4} \int_{B_r(x_0)} |\nabla \psi |^2 - 2 \psi \big(\lambda_1 \I_D - \lambda_2 \I_{D^c} \big) d\mathcal{H}^2 - \frac{2}{r^5} \int_{\p B_r(x_0)} \psi^2 d \mathcal{H}^1
\ee
is monotone increasing. Proposition \ref{P:Class} implies the following alternatives occur, hence we need to consider case by case. \\

\noindent {\bf Blow-up limit I: } 
If $\Phi_{x_0}^\psi(0+) = - \infty$. One has 
\[
\lim_{r\to 0}r^{-5} \int_{\p B_r(x_0)} \psi^2 d\mathcal{H}^1 = + \infty.
\]
Proposition \ref{P:Class} implies that, as $r \rightarrow 0$, any limit function of
\be
\frac{\psi(x_0 + rx)}{S(x_0,r)}
\ee
must be a homogeneous harmonic polynomial of the form
\be
\frac{ x_1x_2+A(x_1^2- x_2^2)}{ \|x_1x_2+A(x_1^2- x_2^2)\|_{L^2(\p B_1)}},\quad A \in \BFR,
\ee
 where $S(x_0,r) = (\frac{1}{r} \int_{\p B_r(x_0)} \psi^2 d \mathcal{H}^1)^{\frac{1}{2}}$. Up to a rotation, we may assume that $A=0$. Let $D_0$ be the limit of $D_r$, then
 $D_0=\{(x_1,x_2)|x_1x_2>0\}$, or $D_0=\{(x_1,x_2)|x_1>0, x_2>0\}$, or $D_0=\{(x_1,x_2)|x_1<0, x_2<0\}$. \\

\noindent {\bf Blow-up limit II: } 
If $\Phi_{x_0}^\psi (0+) > - \infty$, by Proposition \ref{P:Class} one has
\be
\psi_r (x):= \frac{\psi(x_0 + rx)}{r^2}
\ee
is bounded in $W^{1,2}(B_R)$ for any given $R>0$, and there exists a subsequence $\{\psi_{r_j}\}$ weakly converging to degree two homogeneous  function $\psi_0$ satisfying
\be \label{E:elliptic-1}
\begin{cases}
	-\Delta \psi_0 = \lambda_1 \I_{D_*^1} - \lambda_2 \I_{D_*^2},  & \quad x \in \mathbb{R}^2, \\
	\quad \;  \psi_0 = 0, & \quad x \in \p D_*, 
\end{cases}
\ee
where $D_r:=\{x\; |x_0 + r x \in D \}$ and the blow-up limit $\lambda_1 \I_{D_*^1}-\lambda_2 \I_{D_*^2} := \lim_{j \to \infty}\lambda_1 \I_{D_{r_j}}-\lambda_2\I_{(D_{r_j})^c}$ depends on the extraction of the subsequence $r_j \rightarrow 0$. Let $\psi_0 = r^2 f(\theta)$. {Recall the
polar coordinate formula of the Laplacian $\Delta =\p_{rr} + \frac{1}{r} \p_r + \frac{1}{r^2} \p_{\theta \theta}$.
In any given connected component of $D_*^1$, we solve the equation $f''(\theta)+4f(\theta)=-\lambda_1$ and obtain that $f(\theta)=C_1\sin(2\theta+C_2)-\frac{\lambda_1}{4}$ for some constants $C_1$ and $C_2$. Similarly, in a connected component of $D_*^2$, $f(\theta)=C_3\sin(2\theta+C_4)+\frac{\lambda_2}{4}$. Since $\psi\in C^{1,\alpha}$, we have $f\in C^{1,\alpha}.$

If $D_*^1\neq \emptyset$, then it easy to see that for all $x\in \partial D_*^1\cup \partial D_*^2$ and $x\neq 0$, we have $|\nabla u(x)|\neq 0$. Thus $|\mathbb{R}^2\setminus(D_*^1\cup D_*^2)|=0$. 
A direct computation shows that equation \eqref{E:elliptic-1} is equivalent to  the following ODE concerning $f(\theta)$
\be \label{E:angle-1} 
\begin{cases}
	- f^{\prime \prime}(\theta)  - 4 f(\theta) = \lambda_1, \;\; \theta_{2i} \leq \theta \leq \theta_{2i+1}, \\
	- f^{\prime \prime}(\theta)  - 4 f(\theta) =  - \lambda_2,\;\; \theta_{2i+1} \leq \theta  \leq \theta_{2i+2},\\
	f(\theta_{2i})=f(\theta_{2i+1})=f(\theta_{2i+2})=0,
\end{cases}
\ee
for $i=0,\ldots,N$, where $f$ is positive in $(\theta_{2i},\theta_{2i+1})$. \\

Without loss of generality, we let $\theta_0=0$ and $\theta_{2N+2}=2 \pi$. By the standard ODE technique, we can write down the solution  in each sector explicitly as follows:
{\footnotesize
\be \label{E:solu}
f(\theta) = \begin{cases} 
	- \frac{\lambda_1}{4}+ \frac{\lambda_1}{4\sin2(\theta_{2i+1}-\theta_{2i})} \left( (\sin 2\theta_{2i+1} - \sin 2\theta_{2i}) \cos 2 \theta  + (\cos 2\theta_{2i} -\cos 2 \theta_{2i+1}) \sin 2 \theta  \right), \\
		\quad \quad \quad \text{ for }  \theta \in (\theta_{2i},\theta_{2i+1}), \\
	\frac{\lambda_2}{4} - \frac{\lambda_2}{4\sin2(\theta_{2i+2}-\theta_{2i+1})} \left( (\sin 2\theta_{2i+2} - \sin 2\theta_{2i+1}) \cos 2 \theta  + (\cos 2\theta_{2i+1} -\cos 2 \theta_{2i+2}) \sin 2 \theta  \right),\\
	\quad \quad \quad \text{ for } \theta \in (\theta_{2i+1},\theta_{2i+2}).
 \end{cases}
\ee
}
It is easy to see that $\theta_{2i+1}-\theta_{2i+2}\in(0, \frac{\pi}{2})$, otherwise $f$ will change the sign. On the other hand, $f$ can change its sign in $(\theta_{2i+1},\theta_{2i+2})$. Note that $f'(\theta_{2i+1})<0<f'(\theta_{2i+2})$. If $f<0$ in $(\theta_{2i+1},\theta_{2i+2})$, then we have $\theta_{2i+2}-\theta_{2i+1}\in(0,\frac{\pi}{2})$; otherwise it follows $\theta_{2i+2}-\theta_{2i+1}\in (\pi,\frac{3\pi}{2})$. Since $f\in C^{1,\alpha}$, there is
\be \label{E:angle-comb-1}
\begin{split}
& \frac{\lambda_{1}}{2 \sin 2 \alpha_i} (-1 + \cos 2 \alpha_i) = f^\prime(\theta_{2i+1}) = -\frac{\lambda_{2}}{2 \sin 2 \beta_i} (1 - \cos 2 \beta_i), \\
& \alpha_i:=\theta_{2i+1}-\theta_{2i}, \; \beta_i:=\theta_{2i+2} - \theta_{2i+1}.
\end{split}
\ee
It follows   
\be \label{E:angle-comb}
\frac{\tan \alpha_i}{\tan \beta_i} = \frac{\lambda_2}{\lambda_1}>0, i=0,\ldots,N,
\ee
which are independent of the index $i$. Similarly the continuity of $f^\prime(\theta)$  implies 
\be \label{E:angle-comb-2}
\frac{\tan \alpha_{i+1}}{\tan \beta_i} = \frac{\tan \alpha_{i}}{\tan \beta_{i-1}} = \frac{\lambda_2}{\lambda_1}>0, i=1,\ldots,N-1.
\ee

Equalities \eqref{E:angle-comb} and \eqref{E:angle-comb-2} follow these angles $\alpha_i$ coincide with each other and $|\beta_i-\beta_j|\in\{0,\pi\}$ for all $i,j$. Therefore, we have the following alternatives:
\begin{enumerate}
	\item There exists a super-harmonic sector of $\psi_0$ with the opening angle denoted by $\beta >\pi$ without loss of generality. The complement consists of $N$ sectors with opening angle $\alpha$ and other $N-1$ sectors with opening angle $\beta-\pi$ which occur alternatively. Here $\alpha$ and $\beta$ satisfy 
	\[
	\alpha+\beta = \frac{N+1}{N} \pi \; \text{ and } \; \frac{\tan \alpha}{\tan \beta} = \frac{\lambda_2}{\lambda_1}.
	\]  
	In particular, the pattern is $Z_2$-symmetric. 
	\item The pattern is $N$-fold symmetric which super-harmonic and sub-harmonic domains of $\psi_0$, denoted by $\{(r,\theta) \in \R^2 \big| r>0, 0< \theta< \alpha\}$ and $\{(r,\theta) \in R^2 \big| r>0, \alpha<\theta< \alpha+\beta\}$ and their copies by rotating $\frac{N}{2\pi}k, k=1,\ldots,N-1$ respectively, occur alternatively, where $N(\alpha+\beta)=2\pi$.
	\end{enumerate} 
\vspace{0.3 cm}

If $D_*^1=\emptyset$, we can easily show that one of the following holds:
\begin{enumerate}
\item  $D_*^2=\emptyset$, then $\phi_0\equiv 0$.
\item  $\overline{D_*^2}=\mathbb{R}^2$, then $\phi_0(x_1,x_2)=Ax_1^2-Bx_2^2$ after a rotation, where $A>0$ and $B\geq 0$ are constants such that $A-B=\frac{\lambda_2}{2}$.
\item  $D_*^2\neq\emptyset$, $\overline{D_*^2}\neq\mathbb{R}^2$, then $\psi_0(x_1,x_2)=\frac{\lambda_2}{2}(x_1^+)^2$ after a rotation, where $x_1^+$ denotes the positive part of $x_1$.
\end{enumerate} 
Due to the above discussion, we can show
\begin{lemma} \label{L:ODE}
Suppose $x^0 \in S^\psi$ is a singular point. Regarding the blow-up limit (2) in Proposition \ref{P:Class}, for the equation \eqref{E:angle-1} we have the following conclusions: \\
1. If $D_*^1\neq\emptyset$, then
	\begin{enumerate}	
	\item Both of $D_*^1$ and $D_*^2$ consist $N(\geq 3)$ open cones, which are distributed alternately, and $|\mathbb{R}^2\setminus\left(D_*^1\cup D_*^2\right)|=0$.
	
	\item About opening angles of these cones, one of the following two alternatives holds:
	\begin{enumerate}
		\item The opening angles of all cones in $D_*^1$ are $\alpha\in(0,\frac{\pi}{2})$, and the opening angle of all cones in $D_*^2$ is $\beta\in(0,\frac{\pi}{2})$ (Figure \ref{fig:b});
		\item The opening angles of all cones in $D_*^1$ are $\alpha\in(0,\frac{\pi}{2})$,  $N-1$ cones of $D_*^2$ have the same open angle $\beta\in (0,\frac{\pi}{2})$ and the remaining one has an open angle $\beta+\pi$, while $\psi_0$ changes its sign in the cone with an open angle $\beta+\pi$ (Figure \ref{fig:c}).  
	\end{enumerate} 
    $\alpha$ and $\beta$ depend on $\lambda_1$, $\lambda_2$ and $N$.
	\end{enumerate} 
2. If $D_*^1=\emptyset$, then one of the following holds:
\begin{enumerate}
\item  $D_*^2=\emptyset$, then $\psi_0\equiv 0$.
\item  $\overline{D_*^2}=\mathbb{R}^2$, then $\psi_0(x_1,x_2)=Ax_1^2-Bx_2^2$ after a rotation, where $A>0$ and $B\geq 0$ are constants such that $A-B=\frac{\lambda_2}{2}$.
\item  $D_*^2\neq\emptyset$, $\overline{D_*^2}\neq\mathbb{R}^2$, then $\psi_0(x_1,x_2)=\frac{\lambda_2}{2}(x_1^+)^2$ after a rotation.
\end{enumerate} 
\end{lemma} 
\begin{figure}[htbp]   
\centering   

\subfigure[] { \label{fig:b}
	\includegraphics[width=0.232\textwidth]{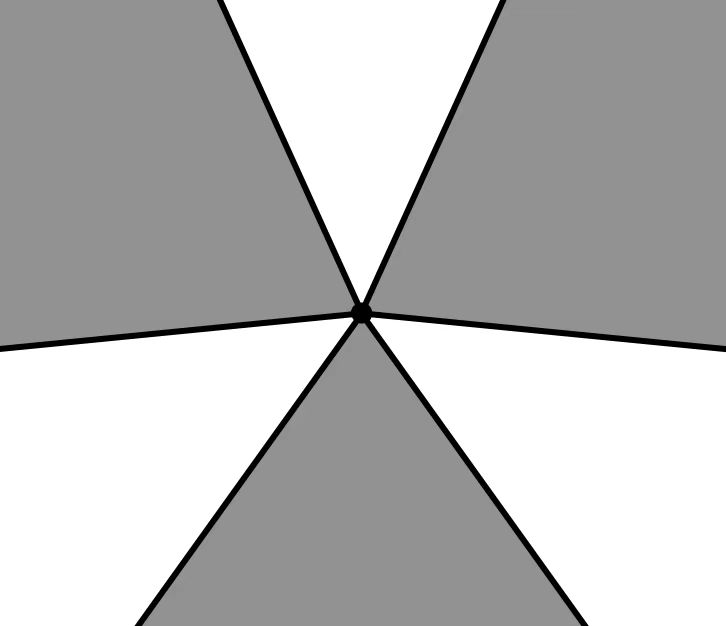}
}		
\subfigure[]{ \label{fig:c}
	\includegraphics[width=0.2\textwidth]{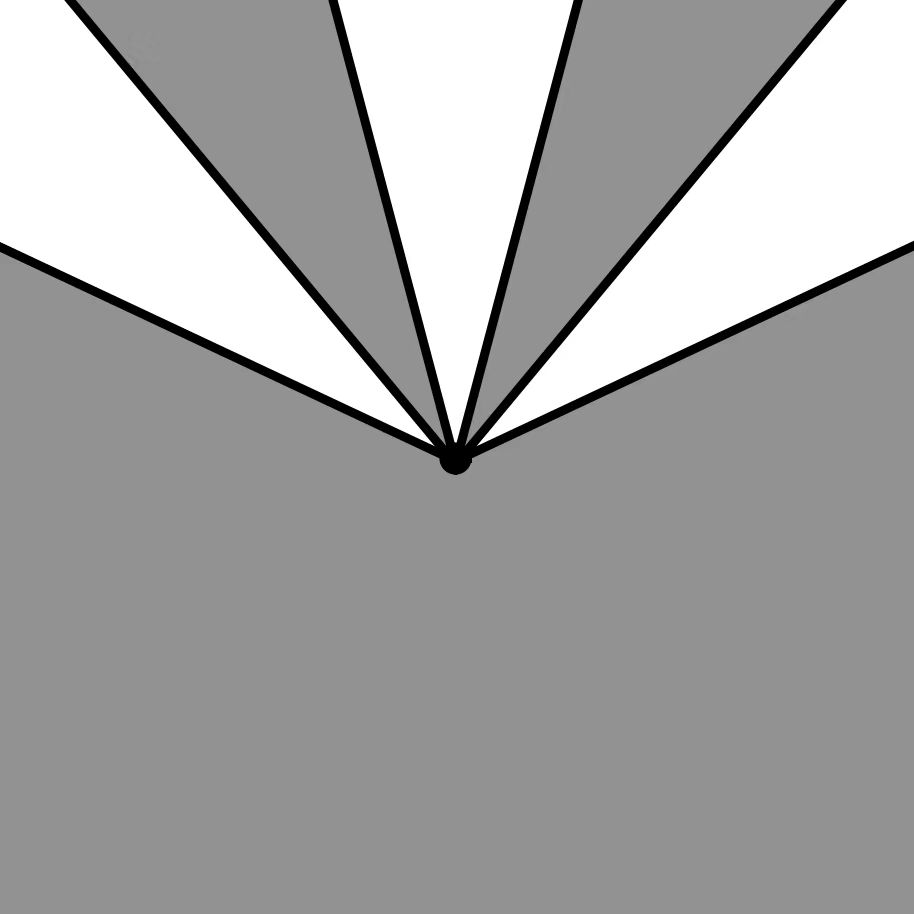}
}		

$D_*^1$: white, $\ \ $$D_*^2$: grey
\caption{Classification of the blow-up limits II}

\end{figure}
\vspace{0.4 cm}
We summarize the above conclusions in the following theorem.
\begin{theorem}\label{T:main-1} 
	Suppose $D \subset \R^n$ is a bounded domain and $(u,D)$ solves the {two-dimensional} unstable elliptic free boundary problem 
	\be 
	\begin{cases}
		-\Delta \psi  =  \lambda_{1}\mathbf{I}_{D} - \lambda_2 \mathbf{I}_{D^c} & \quad \text{ in } \mathbb{R}^2, \\
		\psi  =0 & \quad \text{ on } \p D.
	\end{cases}, \quad \lambda_1,\lambda_2>0.
	\ee
Then the regular part of the free boundary $\p D\setminus \mathcal{S}$ is locally real-analytic, where
\[
\mathcal{S} := \{ x \in \p D\; \big| \;|\nabla \psi(x)| = 0 \}.
\] 	
For any given $x_0\in\mathcal{S}$, one of the following holds,
\begin{enumerate}
    \item $\lim_{r\to 0}r^{-5} \int_{\p B_r(x_0)} u^2 d\mathcal{H}^1 \to +\infty$ as $r\to 0+$,  
    \[
    \phi_r(x): =\frac{\psi(x_0+rx)}{S(x_0,r)}
    \]
    is bounded in $W_{loc}^{1,2}(\mathbb{R}^n)$, and each limit of $\phi_r(x)$
    as $r\to 0$ is a homogeneous harmonic polynomial of degree two, where
    \[
    S(x_0,r):=\left(r^{-1}\int_{\partial B_r(x_0)}\psi^2d\mathcal{H}^{1}\right)^{\frac{1}{2}}.
    \]
    \item  $\limsup_{r\to 0}r^{-5} \int_{\p B_r(x_0)} \psi^2 d\mathcal{H}^{1}<\infty$,  
    \[
    \psi_r(x):=\frac{\psi(x_0+rx)}{r^2},
    \]
    is bounded in $W_{loc}^{1,2}(\mathbb{R}^2)$, and each limit of $\psi_r(x)$ as $r\to $ is a degree two homogeneous function satisfying 
    \begin{equation*}
		\begin{cases}
			-\Delta \psi_{0}=\lambda_{1}\mathbf{I}_{D_*^1}-\lambda_{2}\mathbf{I}_{D_{*}^{2}}  & \text{ in }\mathbb{R}^{n},\\
			\quad \quad \psi_0   =0 & \text{ on }  \partial D_{*}^1 \cup \partial D_*^2,
		\end{cases}
		\end{equation*}
where $D_*^1$ and $D_*^2$ are open cones satisfying $D_*^1\cap D_*^2=\emptyset$, $\mathbb{R}^2\setminus(D_*^1\cup D_*^2)\subset\{\psi_0= 0\}$ and $\psi_0>0$ in $D_*^1$. The complete classification of the solutions $\psi_0$ is given in Lemma \ref{L:ODE}.
\end{enumerate}
\end{theorem}

The proof concerned with the regular portion is standard thus we just outline it here.
	
By the standard elliptic regularity theorem, we have $u \in C^{1,\alpha}$ for $0<\alpha<1$ due to \eqref{E:Sign-ch}. For the free boundary $\p D$ near a regular point $x_0 \in \p D$, i.e., $|\nabla u(x_0)| \neq 0$ and $u(x_0)=0$, by applying the implicit function theorem we have $\p D$ is locally a $C^{1,\alpha}$ curve near $x_0$, namely there exists $r>0$ such that $\p D \cap B_r(x_0)$ is parameterized by the graph of a $C^{1,\alpha}$ function. Following the idea in the proof of Theorem 3.1 of \cite{KNS1978}, we apply the hodograph transform and the local real-analyticity of the curve follows. 

\begin{remark}
    If $D$ is unbounded, under the assumption that $u>0$ in $D$, the same conclusions in Theorem \ref{T:main-1}  also hold.
\end{remark}

Moreover, suppose that the domain $D$ is Lipschitz, then the second case in Theorem \ref{T:main-1} can be ruled out. 
\begin{proposition}\label{C:Singularity-limit}
Under the assumptions of Theorem \ref{T:main-1}, if $D$ is Lipschitz and $x_0\in\partial D$ is a singular point, then the blow-up limit II will not occur.   
\end{proposition} 
\begin{proof}
Suppose towards a contradiction that we have a blow-up limit II at some singular point $x_0\in\partial D$. Without loss of generality we may assume that $x_0=0$ and the boundary of $D$ near $x_0$ can be written as $\{(x_1,x_2)\;|\;x_1\in(-\delta,\delta),x_2=f(x_1)\}$, where $f$ is a Lipschitz function and $L$ is the Lipschitz constant. Then we have $\{(x_1,x_2)\;|\; x_1\in(-\delta,\delta),f(x_1)<x_2<\delta\}\subset D$ and $\{(x_1,x_2)\;|\; x_1\in(-\delta,\delta),-\delta<x_2\leq f(x_1)\}\subset D^c$,  Suppose that $\{\psi_{r_j}\}=\frac{\psi(r_j x)}{r_j^2}$ is a blow-up sequence converging to a function $\psi_0$ that is homogeneous of degree 2, with
\[
-\Delta\psi_0=\lambda_1\I_{D_*^1}+\lambda_2\I_{D_*^2},
\]
where $D_*^1$ and $D_*^2$ are open cones satisfying $D_*^1\cap D_*^2=\emptyset$, $\mathbb{R}^2\setminus(D_*^1\cup D_*^2)\subset\{\phi_0=0\}$ and $\psi_0=0$ on $\partial D_*^1\cup\partial D_*^2$. It is easy to see that $\psi_{r_j}$ satisfies 
\[
-\Delta\psi_{r_j}=\lambda_1\I_{D_j}+\lambda_2\I_{D_j^2},
\]
where $D_j:=\{x\in\mathbb{R}^2:r_j x\in D\}$ and the boundary of $D_j$ can be represented by the graph of a Lipschitz function $f_j$, where $f_j(x_1)=\frac{f(r_j x)}{r_j}$. Moreover, we have
$\{(x_1,x_2)\;|\;x_1\in(-\delta r_j^{-1},\delta r_j^{-1}),f_j(x_1)<x_2<\delta r_j^{-1}\}\subset D_j$ and $\{(x_1,x_2)\;|\; x_1\in(-\delta r_j^{-1},\delta r_j^{-1}),-\delta r_j^{-1}<x_2\leq f_j(x_1)\}\subset D_j^c$. Since $\|f_j\|_{C^{0,1}}\leq L$ for all $j\in\mathbb{N}$, passing if necessary to a subsequence,  $f_j$ converges uniformly to a Lipschitz function $f_0$ with Lipschitz constant $L$. Let $x=(x_1,x_2)$ such that $x_2>f_0(x_1)$. It follows from the uniform convergence of $f_j$ that there exists a constant $\tau>0$ such that $B_\tau(x)\subset D_j$ for all large $j$. Thus $\psi_{r_j}>0$ and $-\Delta \psi_{r_j}=\lambda_1$ in $B_{\tau}(x)$. We conclude that $\psi_0\geq 0$ and $-\Delta\psi_0=\lambda_1$ in $B_{\tau}(x)$. By the maximal principle, we have $\psi_0>0$ in $B_{\tau}(x)$. Since $x$ is arbitrary, it follows that 
\[
\{(x_1,x_2)\;|\; x_2>f_0(x_1)\}\subset D_*^1.
\]
Similarly, we have $\{(x_1,x_2)\;|\; x_2<f_0(x_1)\}\subset D_*^2$.
Therefore 
\[
D_*^1=\{(x_1,x_2)\;|\; x_2>f_0(x_1)\}, D_*^2=\{(x_1,x_2)\;|\;x_2<f_0(x_1)\},
\]
 which leads to a contradiction with Lemma \ref{L:ODE}.
\end{proof}

\begin{corollary}
Suppose that $D$ is a bounded domain in $\mathbb{R}^2$ with Lipschitz boundary, and $(D,\psi)$ is a solution to the uniformly rotating vortex patch problem with $\Omega\in(0, \frac{1}{2})$:
\begin{equation*}
\begin{cases}
	- \Delta \psi = \I_D - 2 \Omega, \quad & x \in \BFR^2, \\ 
	\nabla(\psi - \frac{\Omega}{2}|x|^2) \rightarrow 0, \quad & |x| \rightarrow \infty, \\
	\psi = 0, \quad & x \in \p D.
\end{cases}
\end{equation*}
Then for each $x_0\in \partial D\cap\{|\nabla \psi|=0\}$, we have $$\lim_{r\to 0}r^{-5} \int_{\p B_r(x_0)} \psi^2 d\mathcal{H}^1 \to +\infty$$ as $r\to 0$ and
    \[
    \psi_r(x): =\frac{u(x_0+rx)}{S(x_0,r)}
    \]
    is uniformly bounded in $W_{loc}^{1,2}(\mathbb{R}^2)$. Moreover, up to sequence of $r$, each limit $\psi_0$ of $\psi_r(x)$
    as $r\to 0$ is a homogeneous harmonic polynomial of degree two, where
    \[
    S(x_0,r):=\left(r^{-1}\int_{\partial B_r(x_0)}u^2d\mathcal{H}^1\right)^{\frac{1}{2}}.
    \]
\end{corollary}
}
\vspace{1 cm}

\subsection{Uniform regularity} \label{SS:Unifom}

The Weiss-type monotonicity formula, as well as the classification of the blow-up limits of singular points, provide some geometric information concerning the $0$-level set of the blow-up limit function. In general, the limits may not be unique, i.e., the limiting solution may depend on the choice of subsequence $\{r_j\}_{j \in \mathbb{N}}$. 

For the singular points where $\psi$ has super characteristic growth ({\bf the blow-up limit (I)}), we further prove that the blow-up limit is unique and the uniform regularity of the free boundary $\p D$ near the singularities in this subsection. The situation is similar to Theorem A established in \cite{ASW2010}, where the proof essentially depends on the uniqueness of the generalized Newtonian potential of the function $- \mathbf{I}_{\{x_1x_2>0\}}$. 

It suffices for us to consider the specific equation 
\[
-\Delta \psi = \I_D - 2 \Omega, \quad \psi = 0 \text{ on } \p D 
\]
where $D$ is simply-connected. Because the free boundary near the singular points of blow-up limit (1) is always cross-like, i.e., they consist of sectors with the opening angle $\frac{\pi}{2}$ which is independent of the values $\lambda_1,\lambda_2>0$. 
\begin{lemma} \label{L:Uniform}
	Let $\psi$ be a solution of \eqref{E:VP-1}  and 
	\[
	S^\psi(x^0,r) \geq \frac{r^2}{\delta}
	\]
	for some $0<\delta \leq \delta_0$, $0<r \leq r_0$, $\sup_{B_3(x_0)} |\psi| \leq M$ and $\psi(x^0)= \left|\nabla \psi(x^0)\right| = 0$ then:
	\begin{enumerate}
		\item There exists a second order homogeneous harmonic polynomial $p^{x^0,\psi} =p$ such that for each $\alpha \in (0,\frac{1}{2})$ and each $\beta \in (0,1)$,
		\be
		\left\| \frac{\psi(x^0+sx)}{\sup_{B_s(x^0)} |\psi|} - p \; \; \right\|_{C^{1,\beta}(\overline{B}_1)} \leq C(M,\alpha,\beta) \left( \frac{\delta}{1+\delta\log(\frac{r}{s})} \right)^\alpha.
		\ee
		\item The set $\{\psi=0\} \cap B_{r_0}(x^0)$ consists of two $C^1$-curves intersecting each other at right angles at $x^0$.
	\end{enumerate}
\end{lemma}
It suffices for us to consider the generalized Newtonian potential    
\be \label{E:Poten}
-\Delta z =  \mathbf{I}_{\mathbf{C}} - 2 \Omega, \quad \mathbf{C}:= \{(r,\theta) \subset \BFR^2 \big| r>0, 0< \theta < \frac{\pi}{2}\}
\ee
Note that the right-hand side of \eqref{E:Poten} is of class $L^\infty$ and homogeneous of degree zero, by \cite{KM1996}, $z$ is uniquely representable in the form of 
\[
z(x) = p(x) \log |x| + |x|^2 \phi(\frac{x}{|x|}) + q_1(x), \quad x \in \BFR^2,
\]
where $p$ is a homogeneous harmonic polynomial of degree two, $\phi \in C^1(S^1)$ and $q_1(x) =a \cdot x +b$. Suppose $z(0)=\nabla z(0) =0$. It is sufficient to compute in the form of class $p(x) \log|x| + |x|^2 \phi(\frac{x}{|x|})$. Consider \eqref{E:Poten} 
\[
-(\p_{rr} + \frac{1}{r}\p_r + \frac{1}{r^2} \p_{\theta \theta}) z = \begin{cases}
	1- 2\Omega  \quad 0<\theta<\frac{\pi}{2} \\
	-2 \Omega \quad \frac{\pi}{2}<\theta< 2 \pi 
\end{cases} 
\] 
in the polar coordinate in which $z$ is of the form 
\[
\begin{split}
	z(x_1,x_2) & = \left(A(x_1^2 - x_2^2) + 2Bx_1 x_2\right) \log|x| + |x|^2 \phi(\frac{x}{|x|}) \\
	&= r^2 \log r \left( A \cos 2 \theta + B \sin 2 \theta \right) + r^2 \phi(\theta)   
\end{split}
\]
with constants $A,B \in \BFR$. A straightforward computation shows that 
\[
- 4( A \cos 2 \theta + B \sin 2 \theta) -  \phi^{\prime \prime}(\theta) - 4 \phi(\theta)  = \begin{cases}
	1-2\Omega \quad 0<\theta<\frac{\pi}{2} \\
	-2 \Omega \quad \frac{\pi}{2}<\theta< 2 \pi 
\end{cases}  
\]
Consider the Fourier series 
\[
\phi(\theta) = \frac{A_0}{2} + \sum_{n=1}^\infty A_n \cos n \theta + B_n \sin n \theta. 
\]
For $n \neq 2$, one has 
\[
A_0 = \frac{1-2\Omega}{2} - 3\Omega,\; A_n = \frac{1}{n \pi } \sin n \frac{\pi}{2}, \; B_n =  \frac{1}{n \pi } \left(1 - \cos n \frac{\pi}{2} \right), n \geq 1, n \neq 2
\] 
and
\[
A=0, \;\; B =\frac{1}{2 \pi}.
\]
Note that $r^2(A_2 \cos 2 \theta + B_2 \sin 2 \theta)$ is a harmonic function. We let 
\be \label{E:Potential} 
\begin{split}
	z(r,\theta) & = r^2 \left( \frac{1 - 8 \Omega}{4} + \frac{1}{\pi}\sum_{n \geq 1, n \neq 2} \frac{1}{n} \left(\sin  \frac{n\pi}{2} \cos n \theta + (1-\cos  \frac{n \pi}{2}) \sin n \theta\right) \right) \\
	& + \frac{r^2 \log r}{2\pi} \sin 2 \theta -\frac{1}{4 \pi} r^2 \sin 2 \theta.
\end{split}
\ee

\begin{lemma} \label{L:Gen-pon}
The potential $z$ given in \eqref{E:Potential} is the unique solution of equation \eqref{E:Poten} satisfying 
\begin{enumerate}
\item $z(0)=|\nabla z(0)|=0$,
\item $\lim_{x \rightarrow \infty} \frac{z(x)}{|x|^3}=0$,
\item $\Pi(z)=0$,
\item $\Pi(z_{\frac{1}{2}}) = \frac{\log 2}{\pi} x_1 x_2$, 
\item $\tau(z_{\frac{1}{2}}) = \frac{\log 2}{ \pi}$,
\end{enumerate} 
where $\mathbb{P}$ is the space of second-order homogeneous harmonic polynomials and $\Pi:W^{2,2}(B_1)\rightarrow \mathbb{P}$ is a projection such that $\Pi(u)$ is the unique minimizer of  
\[
p \rightarrow \int_{B_1} \left| D^2 v -D^2 p \right|^2
\]  
on $\mathbb{P}$ and $\tau(u) \geq 0$ is defined by $\Pi(u) = \tau(u) p, \; p \in \mathbb{P}, \;\;\sup_{B_1}|p|=1$. 
\end{lemma} 

\begin{proof}
Suppose $h:=\Pi(z)$. Since $h$ is harmonic, we let
\[
D^2 h = \begin{pmatrix} 
	a & b \\
	b & -a 
	\end{pmatrix}. 
\] 
Since $h$ is the minimizer of the mapping $p \rightarrow \int_{B_1} \left| D^2 v -D^2 p \right|^2$, we have 
\[
\begin{split}
0 & = \p_b  \int_{B_1} \left| D^2 v -D^2 h \right|^2 = 4b - 4 \int_{B_1} \p_{12} z \\
& = 4b - \int_{0}^1 \int_{S^1} \bigg(\frac{\sin2 \theta}{2}\left(\p_{rr} z + \frac{1}{r^2} \p_{\theta \theta} \right) z(r,\theta)- \frac{1}{r} \p_{r \theta} z(r,\theta) \\
& + \frac{1}{r^2} \p_\theta z(r,\theta) + \frac{\sin 2 \theta }{2r}\p_r z(r,\theta) \bigg) r dr d \theta. 
\end{split}
\] 
Being aware of that only $\frac{r^2 \log r}{2\pi} \sin 2 \theta -\frac{1}{4 \pi} r^2 \sin 2 \theta$ provides non-zero integrals, we have $b=0$ since 
\[
\begin{split}
& \int_{B_1} \p_{12} \left(\frac{\log(x_1^2 +x_2^2)}{2\pi} x_1 x_2 -\frac{1}{2\pi} x_1x_2 \right) dx_1 dx_2\\
&  = - \frac{1}{2} + \frac{1}{2 \pi} \int_{B_1} \left(2 + \log(x_1^2 +x_2^2) - \frac{4x_1^2x_2^2}{(x_1^2 +x_2^2)^2} \right) dx_1 dx_2 =0, 
\end{split} 
\]
while $a=0$ is obtained similarly by considering 
\[
0 = \p_a  \int_{B_1} \left| D^2 v -D^2 h \right|^2 = 4a. 
\]
Therefore one has $h:=\Pi(z)=0$ due to $h$ being a second-order polynomial. Moreover, we have 
\[
\frac{z(sx_1, sx_2)}{s^2}=z(x_1,x_2) - \frac{x_1 x_2}{\pi} \log s. 
\]
Then it follows 
\[
\Pi(z_{\frac{1}{2}}) = \Pi(z) - \Pi \left( \frac{x_1x_2 \log\left(\frac{1}{2}\right)}{\pi} \right) = \frac{\log 2}{\pi} x_1 x_2.
\]

If $z_1$ and $z_2$ are solutions of \eqref{E:Potential}, the $\eta:=z_1 -z_2$ is a second-order harmonic polynomial then $|\nabla \eta (0)|= \eta(0)=0$ and $\Pi(\eta)=0$ imply $z_1=z_2$,  namely $z$ is the unique solution. 

\end{proof}

\begin{lemma} 
Suppose $\psi$ solves \eqref{E:VP-1}, $x_0 \in \mathcal{S}^u$ and  $\sup_{B_{3}(x_0)} |u| \leq M <+\infty$. Then 
\[
\left( \int_{B_1} \left| D^2 \frac{u(x^0+rx)}{r^2} - D^2 \Pi \left( \frac{u(x^0+rx)}{r^2} \right) \right|^p \right)^{\frac{1}{p}} \leq C(M,p)
\]
and 
\[
\left\| \frac{u(x^0+rx)}{r^2} - \Pi\left(\frac{u(x^0+rx)}{r^2} \right) \right\|_{C^{1,\beta}(\overline{B}_1)} \leq C(M,\beta). 
\]
\end{lemma}

\begin{lemma}\label{L:Control}
For each $\ep>0, M<+\infty, \alpha\in [1,+\infty)$ and $\beta \in (0,1)$, there exist $r_0, \delta>0$ with the following property: suppose that $0<r \leq r_0$ and that $\psi$ is a solution of \eqref{E:VP-1} satisfying $\sup_{B_2(x_0)} |u| \leq M$, $u(x)=|\nabla u(x)|=0$ and 
\[
\mathcal{L}^2\left( \left(\{u(x+r \cdot)>0 \} \Delta \{x_1x_2>0\} \right) \cap B_1 \right) \leq \delta. 
\] 
Then 
\be
\left\| \frac{u(x+r \cdot)}{r^2} - \Pi \left( \frac{u(x+r \cdot)}{r^2} \right) - z \right\|_{C^{1,\beta}(\overline{B}_1)} \leq \ep. 
\ee
\end{lemma} 
\begin{proof}
Suppose that $r_j \rightarrow 0$ such that 
\[
\mathcal{L}^n \left(\{u(x+r \cdot)>0 \} \Delta \{x_1x_2>0\} \right) \rightarrow 0 \;\; \text{ as } j \rightarrow \infty
\]
and
\[
\frac{u_j(x^j+r_j \cdot)}{r_j^2} - \Pi\left( \frac{u_j(x^j+r_j \cdot)}{r_j^2} \right) \rightarrow \wt z \;\text{ strongly in } C_{loc}^{1,\beta}(\BFR^n) \text{ and  weakly in } W_{loc}^{2,\alpha}(\BFR^n).
\]
Now consider $\wt N$ be the Newtonian potential of $\Delta u_j$, i.e.,
\[
\wt N(y):= \frac{1}{2\pi} \int_{\BFR^2} \log|y - \xi| \Delta u_j(\xi) d \xi. 
\]
Let $N(y) = \wt N(y) - \wt N(x^j) - \nabla \wt N(x^j)(y-x^j)$, then $h(y) = u_j(y) - N(y)$ is a harmonic function. We have 
\[
\left|u_j(y) - N(y) - \mathcal{D}^2 h(x^j)(y-x^j)(y-x^j) \right| \leq C \left|x^j - y \right|^3 \text{ in } B_{\frac{1}{2}}(x^j). 
\]
For the scaled function $v_j(y) = \frac{\psi_j(x^j+r_jy)}{r_j^2},\; N_j(y) = \frac{N(x^j+r_j y)}{r_j^2}$ and $p_j(y) = \mathcal{D}^2h(x^j)y\cdot y$, we obtain 
\[
\left| v_j(y)-N_j(y) -p_j(y) \right| \leq C r_j |y|^3  \text{ in } B_{\frac{1}{2}}(x^j). 
\]
Thus 
\[
v_j - \Pi(v_j) = N_j - \Pi(N_j) + o(1), \;\; j \rightarrow \infty. 
\]
Therefore $N_j$ converges locally to $N_0$, where 
\[
-\Delta N_0 = \mathbf{I}_C - 2 \Omega, \;\; N_0(0)=\nabla N_0(0)=0, \text{ and } N_0 - \Pi(N_0) = \wt z.
\]
Then it suffices to show that $N_0(y)=o(|y|^3)$ as $|y| \rightarrow \infty$ since one would conclude $\wt z = N_0 - \Pi(N_0) = z$ due to the uniqueness result given in Lemma \ref{L:Gen-pon}. Since $\mathcal{D}^2 N_0 \in BMO$ so that 
\[
\int_{B_1} \left| \frac{\mathcal{D}^2\left(N_0(Ry)\right) - \overline{\mathcal{D}^2\left(N_0(R\cdot)\right)}}{\sup_{B_R} \left|\mathcal{D}^2 N_0 \right|} \right|^2 d y \leq C \frac{R^4}{\sup_{B_R} \left|\mathcal{D}^2 N_0 \right|^2}
\]
for all $R \in (0,+\infty)$, where $\overline{\mathcal{D}^2\left(N_0(R\cdot)\right)}$ denotes the mean value of $\mathcal{D}^2(N_0(R\cdot))$ on $B_1$. Therefore  we have a sun-sequence $\frac{N_0(R_k\cdot)}{\sup_{B_{R_k}} \left|\mathcal{D}^2 N_0 \right|}$  converges to a  degree $2$ homogeneous harmonic polynomial as $R_k \rightarrow \infty$. Assume  
\[
\limsup_{|y| \rightarrow \infty} \frac{|N_0(y)|}{|y|^3} >0.
\] 
Since $\Delta(N_0-z) = 0$ and 
\[
\limsup_{|y| \rightarrow \infty} \frac{|N_0(y)-z(y)|}{|y|^3} >0,
\] 
$N_0-z$ must be a harmonic polynomial of degree $m \geq 3$ which is a contradiction.
\end{proof}
Based on Lemma \ref{L:Control}, Lemma \ref{L:Uniform} is proved by the same approach as Theorem A in \cite{ASW2010},  hence we omit it here. \\

Now we are in the position to prove the main theorem.


\begin{proof}[Proof of Theorem \ref{T:main-2}]

Employ Theorem \ref{T:main-1} by letting $\lambda_1=1-2\Omega$ and $\lambda_2=2 \Omega$. Suppose that there is no stagnation point on the boundary, namely $|\nabla \psi| \neq 0$ on $\p D$. Since $\p D$ is a connected component of level set $\{\psi =0\}$ and $\psi \in C^{1,\alpha}$ due to standard elliptic regularity theorem,  $\p D$ is a $C^{1,\alpha}$ curves, which is indeed locally real-analytic due to Theorem \ref{T:main-1}. Since the curve $\p D$ is real-analytic hence $C^\infty$,  $\p D$ is a $C^\infty$ smooth manifold due to the fundamental theorem concerning partition of unity.

For the singular set, by Proposition \ref{C:Singularity-limit} we have that only Type I singular points can occur. Moreover, suppose $x^0 \in S^\psi \subset \p D$, i.e., $\nabla \psi(x_0) = 0$, by Theorem \ref{T:main-1} and Lemma \ref{L:Uniform}, there exists a small parameter $r_0>0$ such that $\{\psi =0 \} \cap B_{r_0}(x_0)$ consists of two $C^1$-curves intersecting each other at right angles at $x^0$, meaning that $x^0$ must be  an isolated singular point.

In the last part, we will show that the singular set is finite, or equivalently the isolatedness of singular points  since $\p D$ is rectifiable. Assume the conclusion fails, i.e., the singular points are not isolated, there would be a sequence of singular points
\[
\{x_j\}_{j=1}^\infty \in S^\psi:=\left\{x \in \p D \; \big|\, \psi(x_j)=|\nabla \psi(x_j)| = 0  \right\},
\]
converging to $x^0 \in \p D$ satisfying $\nabla \psi(x^0)=0$ by taking a subsequence since $S^\psi$ is a closed set. 

We note that $x^0$ must be a singular point, i.e., for any $r>0$, $\p D \cap B_r(x^0)$ can not represented as the graph of a smooth function. Indeed, Theorem \ref{T:main-1} establishes the classification for all critical points on $\p D$ satisfying $\nabla \psi(x^0) =0$. By Proposition \ref{C:Singularity-limit}, only blow-up limit I can occur. Then Lemma \ref{L:Uniform} implies that there exists $r_0>0$ such that $\p D \cap B_{r_0}(x^0)$ is a crossing of two $C^1$-curves thus $x_j \notin \p D \cap B_{r_0}(x^0)$. Therefore,  $x^0$ is a singular point on the boundary, and $x_j \notin \cap B_{r_0}(x^0)$, which contradicts with the assumption $x_j \to x^0$. The proof is then complete.

\end{proof}

{\bf Conflict of interest statement.} On behalf of all authors, the corresponding author states that there is no conflict of interest.

{\bf Data availability statement.} No new data were created or analysed in this article. Data sharing is not applicable to this article.

\bibliographystyle{plain}

\end{document}